\documentclass[11pt]{article}

\usepackage{graphicx}
\usepackage{amssymb,amsthm,longtable}
\usepackage{color}
\theoremstyle{definition}

\theoremstyle{plain}

\newtheorem{proposition}{Proposition}

\newcommand{\ZZ}{\mathbb{Z}}

\title{Determination of division algebras with 243 elements}
\author{ I.F. R\'
ua\thanks{Departamento de Matem\'aticas, Universidad de Oviedo, rua@uniovi.es .}\and El\'\i as F. Combarro\thanks{Artificial Ingelligence Center,
University of Oviedo, \{elias,ranilla\}@aic.uniovi.es . }\and {J. Ranilla$^\dag$}}
\date{}

\begin{document}

\maketitle

\begin{abstract}
Finite nonassociative division algebras (i.e., finite semifields) with 243 elements are completely classified.
   
\end{abstract}

\section{Introduction}

Finite division rings, also known as \textbf{finite semifields}, are nonassociative rings with identity such that the set of nonzero elements is closed under the product (i.e., a loop \cite{Knu65,Cor99}). In case it has no identity they are known as \textbf{presemifields}. These objects have been studied in different contexts: finite
geometries (they coordinatize projective semifield planes \cite{Hall}), coding theory \cite{Calderbank3,codigosdesemicuerpos,Symplectic}, combinatorics and graph theory \cite{ISSAC}.

Computational
methods have been considered in the study of these objects. Among others, the classification of finite semifields of orders 16 \cite{Kleinfeld,Knu65}, 32 \cite{Walker,Knu65} and, more recently, orders 64 \cite{Rua} and 81 \cite{Dempwolff} have been obtained with the help of computational tools. Presently, only the cases of orders 128, 243 and 256 remain to achieve the classification of semifield planes of order 256 or less suggested in \cite{Kantor}. 

In this paper we present a classification of semifields with \textbf{243 elements} up to isotopy. It is a computer-assited classification based on the algorithms introduced in \cite{Rua}.

\section{Preliminaries}

In this section we collect some definitions and facts on finite
semifields, presemifields and planar functions (see, for instance \cite{Cor99,Knu65}). We restrict ourselves to the particular case of order $243=3^5$. The characteristic of a finite presemifield $D$ with $3^5$ elements is $3$, and $D$ is a $5-$dimensional algebra over $\mathbb{F}_3$.
    If $D$ is a semifield, then $\mathbb{F}_3$ can be chosen to be contained in its associative-commutative center $Z(D)$. Other relevant subsets of a finite semifield are the left, right, and middle nuclei ($N_l,N_r,N_m$), and the nucleus $N$ which have to be field extensions $\mathbb{F}_{3^e}$ ($e\le 5$).

Classification of presemifields is usually considered up to \emph{isotopy} (since this corresponds to classification of the corresponding projective planes up to isomorphism):
	If $D_1,D_2$ are two presemifields of order $3^5$, an {isotopy} between $D_1$ and $D_2$ is a triple $(F,G,H)$ of bijective $\mathbb{F}_3-$linear maps $D_1\to D_2$  such that
	$$H(ab)=F(a)G(b)\;,\; \forall a,b\in D_1.$$
Any presemifield is isotopic to a finite semifield.

If $\mathcal B=[x_1,\dots,x_5]$ is a $\mathbb{F}_3$-basis of a presemifield $D$, then there exists a unique set of constants $\mathbf A_{D,\mathcal B}=\{A_{i_1i_2i_3}\}_{i_1,i_2,i_3=1}^5\subseteq \mathbb{F}_3$ such that 
$$x_{i_1}x_{i_2}=\sum_{i_3=1}^5{A_{i_1i_2i_3}}x_{i_3}\; \forall i_1,i_2\in\{1,\dots,5\}$$
This set of constants is known as \textbf{cubical array} or \textbf{3-cube} corresponding to $D$ with respect to the basis $\mathcal B$, and it completely determines the multiplication in $D$. If $D$ is a presemifield, and $\sigma\in S_3$ (the symmetric group on the set $\{1,2,3\}$), then the set
$$\mathbf A_{D,\mathcal B}^\sigma=\{A_{i_{\sigma(1)}i_{\sigma(2)}i_{\sigma(3)}}\}_{i_1,i_2,i_3=1}^5\subseteq \mathbb{F}_3$$ is the 3-cube of a presemifield. 
Different choices of bases lead to isotopic presemifields.
Up to 
six projective planes can be constructed from a given finite semifield 
$D$ using the transformations of the group $S_3$. So, the classification of finite semifields can be reduced to the classification of the corresponding projective planes up to the action of the group $S_3$.

Finite semifields of order 243 can be constructed from sets of matrices with certain properties \cite{Dempwolff}\cite[Proposition 3]{Hentzel}.

\begin{proposition}\label{matrices}
    There exists a finite semifield $D$ of order $3^5$ if, and only if, there exists a
    set of $5$ matrices (a \textbf{standard basis} of $D$) $S_D=\{A_1,\dots,A_5\}\subseteq GL(5,3)$ (the set of invertible matrices of size $5$ over $\mathbb{F}_{3}$) such
    that:
    \begin{enumerate}
        \item $A_1$ is the identity matrix $I$;
        \item $\sum_{i=1}^5\lambda_i A_i\in GL(5,3)$, for all
        {non-zero tuples}
        $(\lambda_1,\dots,\lambda_5)\in \mathbb{F}_3^5$, that is,
        $(\lambda_1,\dots,\lambda_5)\not=\{\overrightarrow 0\}$.
        \item The first column of the matrix $A_i$ is the column vector
        $e_i^\downarrow$ with a $1$ in
    the $i$-th position, and $0$ everywhere else.
    \end{enumerate}
    In such a case, the set $\{B_{ijk}\}_{i,j,k=1}^d$, where $B_{ijk}=(A_j)_{ik}$, is the 3-cube corresponding to $D$ with respect to the canonical basis of $\mathbb{F}_3^5$.
\end{proposition}

If we identify the elements of $\mathbb{F}_3$ with the natural numbers $\{0,1,2\}$, then  {we can use the following convention to represent a semifield $D$} of order $3^5$. Let $S_D=\{A_1,\dots,A_5\}$ be one of its standard bases. Since the first column of $A_i$ has always a one in the $i$-th position and zeroes elsewhere, we can encode $A_i$ as the natural number $\sum_{j=0}^{19}a_j3^j$, where
$$\left(%
\begin{array}{c|ccccc}
 & a_{19} &  a_{14}& a_{9} & a_{4} \\
 &  a_{18} &  a_{13}& a_{8} & a_{3} \\
   e_i^\downarrow& a_{17} &  a_{12}& a_{7} & a_{2} \\
     &  a_{16} &  a_{11}& a_{6} & a_{1} \\
  & a_{15} & a_{10} & a_5 & a_0 \\
\end{array}%
\right)$$

For a concrete representation 
of the semifield one can identify the semifield with $\mathbb{F}_3^5$, and the multiplication with $x*y= {\sum_{i=1}^5} x_i A_i y$, i.e., $A_i$ is the matrix of left multiplication by the element $e_i$, where $\{e_1,\dots,e_5\}$ is the canonical basis of $\mathbb{F}_3^5$. So, the elements of the standard basis are simply coordinate matrices of the linear maps $L_{e_i}:D\to D$, $L_{e_i}(y)=e_i*y$.

%\begin{remark}\label{nolinear}
%	Because of \cite{Alb60}[Lemma 5], the characteristic polynomial of any matrix in $\Sigma\setminus\{0_6,I_6\}$ has no linear factors. We will say that a monic polynomial is \emph{admissible} if it has no linear factors. On the other hand, if a semifield of order 64 is primitive \cite{Wene91} then it has a standard basis such that $A_2$ is a companion matrix whose characteristic polynomial is primitive (i.e. its multiplicative order is 63) \cite[Proposition 2, Corollary 1]{Hentzel}. 
%\end{remark}

The following result shows that, in order to classify finite semifields of order 243, it is possible to impose for extra conditions on the standard bases.

\begin{proposition}
  {Let $D$ be a finite semifield of order 243. Then, there exists an isotope $D'$ of $D$ such that, if $S_{D'}=\{A_1,\dots,A_5\}$ is a standard basis of $D'$, then $A_2$ has one of the following forms:}
  \begin{equation}\label{primera}C(x^5+x^3+x+1)\hbox{ , }C(x^5+2x+1)\hbox{ , }C(x^5+x^3+x+2)\hbox{ , }C(x^5+2x+2)\end{equation}
  or
  \begin{equation}\label{segunda}\left(\begin{array}{cc}C(x^3+x^2+x+2)&0\\0&C(x^2+2x+2)\end{array}\right)\hbox{ , }\left(\begin{array}{cc}C(x^3+2x^2+x+1)&0\\0&C(x^2+x+2)\end{array}\right)\end{equation}
  where $C(p(x))$ is the companion matrix of the polynomial $p(x)\in \mathbb{F}_3[x]$.
\end{proposition}

\begin{proof}
  {We first show that there exists an element $b\in D\setminus \mathbb{F}_3$ (i.e., \emph{non-scalar}) such that the characteristic polynomial of the linear transformation $L_b:D\to D$ ($L_b(x)=b*x$) is $x^5+a_2x^3+a_4x+a_5\in \mathbb{F}_3[x]$.}
{We fix $\{x_1,x_2,x_3,x_4,x_5\}$, a $\mathbb{F}_3$-basis of $D$, and consider the characteristic polynomial of $L_b$ for a generic element $b=\lambda_1x_1+\lambda_2x_2+\lambda_3x_3+\lambda_4x_4+\lambda_5x_5\in D$:}
$$ {x^5+\rho_1(\overline\lambda)x^4+\rho_2(\overline \lambda)x^3+\rho_3(\overline \lambda)x^2+\rho_4(\overline \lambda)x+\rho_5(\overline\lambda)}$$
  {where $\rho_i(\overline \lambda)$ is a homogeneous polynomial in $\mathbb{F}_3[\overline \lambda]=\mathbb{F}_3[\lambda_1,\lambda_2,\lambda_3,\lambda_4,\lambda_5]$, of degree $i$. Consider the system of equations $\rho_1(\overline \lambda)=\rho_3(\overline \lambda)=0$. From the Chevalley-Warning theorem \cite[Theorem 6.6]{Lidl} it has a nonzero solution, i.e., there exists a nonzero element $b\in D$ such that the characteristic polynomial of $L_b$ is of the claimed form.} 
  
  The element $b$ can not be $1$ or $2$, since the trace of the polynomials $(x-1)^5$ and $(x-2)^5$ is not zero, and so $b$ is non-scalar. Because of \cite{Alb60}[Lemma 5], the characteristic polynomial of $L_b$ has no linear factors, and so it has to be one of the following six polynomials:
  $$x^5+x^3+x+1\hbox{ , }x^5+2x+1\hbox{ , }x^5+x^3+x+2\hbox{ , }x^5+2x+2\hbox{ , }x^5+2x^3+1\hbox{ , }x^5+2x^3+2$$
  
  The first four polynomials are irreducible, and so the set $\{1,b,b^2,b*b^2,b*(b*b^{(2})\}$ is a $\mathbb{F}_3$-basis of $D$ (\cite[Section 3]{Alb60}, there on the right, here on the left). This provides a standard basis $S_D=\{[L_1],[L_b],[L_{b^2}],[L_{b*b^2}],[L_{b*(b*b^{(2})}]\}$ where $[L_b]$ has one of the  forms in equation (\ref{primera}) (notice that $D$ is an isotope of itself).

On the other hand, the last two polynomials have the following factorizations:
$$(x^3+x^2+x+2)(x^2+2x+2)\hbox{  , }(x^3+2x^2+x+1)(x^2+x+2)$$
  Because of \cite[Lemma 5]{Alb60}, there exists an isotope $D'$ of $D$ and an element $c\in D'$ such that its \emph{minimal function} is the first factor. So, because both factors are croprime, we can choose a $\mathbb{F}_3$-basis of $D'$ of the form $\{1,c,c^2,d,c*d\}$. In the corresponding standard basis $S_{D'}=\{[L_1],[L_c],[L_{c^2}],[L_{d}],[L_{c*d}]\}$, the matrix $[L_c]$ has one of the forms in equation (\ref{segunda}).
  
\end{proof}

\section{The Semifield Planes of order 243: a classification}

We obtained a complete classification of finite semifields of order 243 with the help of the algorithm introduced in \cite{Rua}. This algorithm searchs for standard bases of division algebras with 243 elements, and classify them according to equivalent $S_3$-equivalent semifields. This is done either for partial or for complete standard bases.

Our algorithm was processed in parallel in Magerit, a cluster of 1204 nodes eServer BladeCenter (1036 JS20 and
168 JS21, both PowerPC 64 bits). Each JS20 node has two processors IBM
PowerPC single-core 970FX (two cores) with 2.2 GHz, 4 GB of RAM and 40
GB of local hard disk. On the other hand, each JS21 node has two
processors IBM PowerPC dual-core 970FX (four cores) with 2.2 GHz, 8 GB
of RAM and 80 GB of local hard disk.
It was installed in 2006 and reached the 9th fastest in Europe and the
34th in the world (Top 500: List from November 2006). In May 2008 it
was upgraded to reach 16 TFLOPS. This powerful cluster has allowed us to fill the gap between the commutative and the nonconmutative case.

Next we present the results obtained from our classification (Table \ref{tabla1}). Let us compare the number of $S_3$-equivalence classes, semifield planes, and coordinatizing finite semifields which were found, with those previously known \cite{conmutativo}.

\begin{table}[htdp]
\begin{center}
\begin{tabular}{|c|c|c|c|}
  \hline
   \textbf{Number of classes}  & \textbf{$S_3$-action} & \textbf{Isotopy} & \textbf{Isomorphism}\\
  \hline
  Previously known &  7&19&27313 \\

  \hline
   Actual number & 9&23&85877 \\
   \hline
\end{tabular}
\caption{{Number of division algebras with 243 elements}}\label{tabla1}
\end{center}
\end{table}%

As we can see two new $S_3$-classes exist, that can not be constructed from commutative semifields. {And four new semifield planes of such an order appear.} Next we present standard bases of these classes ($A_1$ is always the identity matrix) (Table \ref{tabla2}).

\begin{table}[htdp]
\begin{center}
\begin{tabular}{|c|c|c|c|c|c|}
  \hline
  $\#$ & $A_2$  & $A_3$ & $A_4$ &  $A_5$ & \# Semifield\\
  \hline
I       &       129317742& 43151760& 25524498& 2715668620 
 &$\mathbb{F}_{3^5}$\\\hline
II      &       129317638 &44994959 &28587138 &1226007534& Albert's twisted field\\\hline
III     &       129317781 &52757047 &20739470 &3274303432 &Albert's twisted field\\\hline
IV      &       129317742 &43393513 &26923067 &2713804376   & Coulter-Matthews'\\\hline
V       &       129317742 &43215002 &26537147 &2719346408   & Ding-Yuan's\\\hline
VI      &       129317742 &43185096 &19259172 &2718371119 & \cite{conmutativo}\\\hline
VII      &       129317742 &43215002 &26558192 &2719382129 & \cite{conmutativo}\\\hline
VIII      &       129317636 &14673002& 1139489406& 3073918154 &-
\\\hline
IX     &       129317636 &18089998 &3416237282 &1030364558 & -\\\hline
\end{tabular}
 \caption{{Standard bases of division algebras with $243$ elements (VIII and IX are new semifields)}}\label{tabla2}\end{center}
\end{table}

For these semifields, we have computed some information (see \cite{Rua} for the notation and details). Namely, the order of their center and nuclei, $ZN=(Z,N,N_l,N_m,N_r)$, the list of all principal isotopes, and the order of their isomorphism groups. The length of the orbits in the fundamental triangle $(L_x,L_\infty,L_y)$ was also computed given in the form $\sum_{i=1}^ra_i[b_i]$, if $a_i$ cycles of length $b_i$ ($i=1,\dots,r$) exist. Also, where possible, some information on the autotopism group has been include (Table \ref{numerazos}).

\begin{table}
\begin{tabular}{|c|c|c|c|c|c|}
  \hline
   \textbf{Plane}  & $\mathbf{S_3-class}$ & \textbf{$|$At$|$} & $\mathbf{(L_x,L_\infty,L_y)}$&\textbf{S/A sum}& $\mathbf{ZN}$\\
  \hline
  \textbf{I} 	&$\begin{array}{c}\includegraphics[width=1cm,height=1cm]{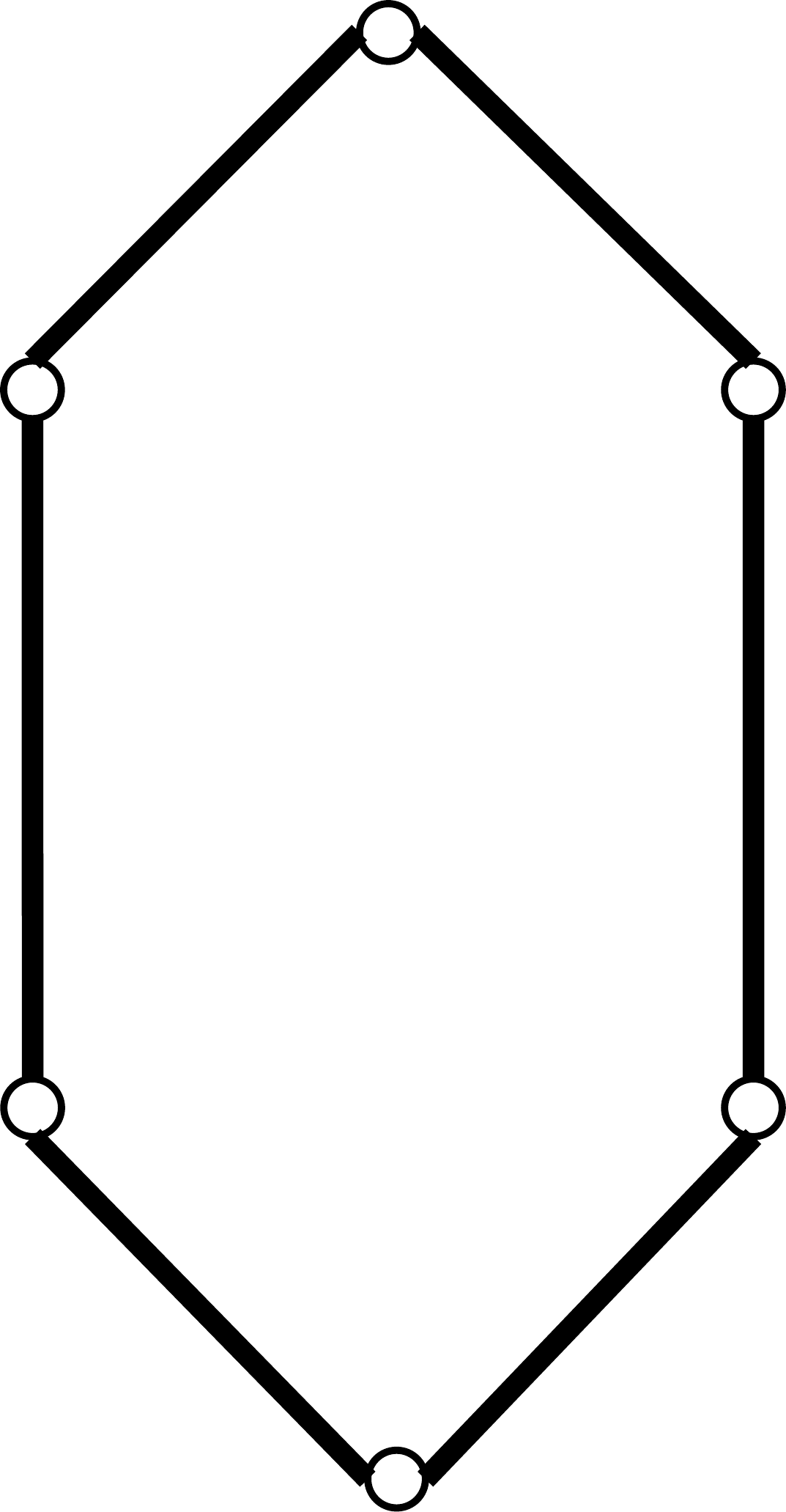}\\\end{array}$&	
  $292820$ 
  & $\begin{array}{c}    2[1]+1[242]\\    2[1]+1[242]\\    2[1]+1[242]\\\end{array}$  	 & $\frac{1}{5}$	& $(243,243,243,243,243)$\\\hline
\textbf{II} 	&$\begin{array}{c}\includegraphics[width=1cm,height=1cm]{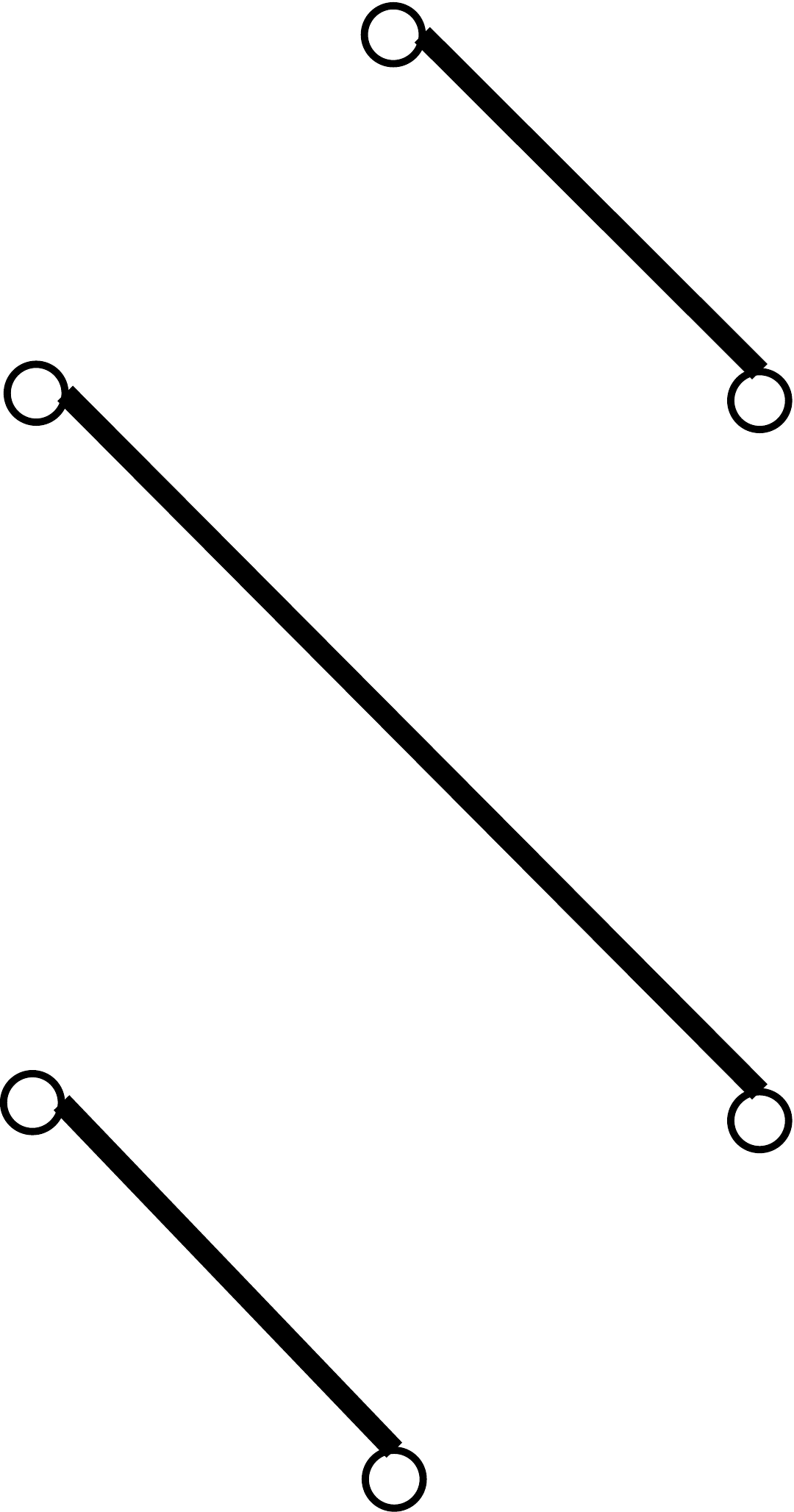}\\\end{array}$&	$  \begin{array}{c} 2420 \\ \hbox{Solvable}\\\end{array}$ & $\begin{array}{c}    2[1]+1[242]\\    2[1]+1[242]\\    2[1]+1[242]\\\end{array}$	 & $\frac{24}{1}+\frac{1}{5}$	& $(3,3,3,3,3)$\\\hline
\textbf{III} 	&$\begin{array}{c}\includegraphics[width=1cm,height=1cm]{hexagonLinGr.pdf}\\\end{array}$&	$  \begin{array}{c} 2420 \\ \hbox{Solvable}\\\end{array}$ & $\begin{array}{c}    2[1]+1[242]\\    2[1]+1[242]\\    2[1]+1[242]\\\end{array}$ & $\frac{24}{1}+\frac{1}{5}$	& $(3,3,3,3,3)$\\\hline
\textbf{IV} 	&$\begin{array}{c}\includegraphics[width=1cm,height=1cm]{hexagonLinGr.pdf}\\\end{array}$&	$  \begin{array}{c} 20 \\ \ZZ_2\times \ZZ_{10}\\\end{array}$ & $\begin{array}{c} 2[1]+1[2]+24[10]\\ 2[1]+1[2]+24[10]\\2[1]+1[2]+24[10]\\\end{array}$	 & $\frac{2928}{1}+\frac{1}{5}$	& $(3,3,3,3,3)$\\\hline
\textbf{V} 	&$\begin{array}{c}\includegraphics[width=1cm,height=1cm]{hexagonLinGr.pdf}\\\end{array}$&	$  \begin{array}{c} 20 \\ \ZZ_2\times \ZZ_{10}\\\end{array}$ & $\begin{array}{c} 2[1]+1[2]+24[10]\\ 2[1]+1[2]+24[10]\\2[1]+1[2]+24[10]\\\end{array}$	 & $\frac{2928}{1}+\frac{1}{5}$	& $(3,3,3,3,3)$\\\hline
\textbf{VI} 	&$\begin{array}{c}\includegraphics[width=1cm,height=1cm]{hexagonLinGr.pdf}\\\end{array}$&	$  \begin{array}{c} 20 \\ \ZZ_2\times \ZZ_{10}\\\end{array}$ & $\begin{array}{c} 2[1]+1[2]+24[10]\\ 2[1]+1[2]+24[10]\\2[1]+1[2]+24[10]\\\end{array}$	 & $\frac{2928}{1}+\frac{1}{5}$	& $(3,3,3,3,3)$\\\hline
\textbf{VII} 	&$\begin{array}{c}\includegraphics[width=1cm,height=1cm]{hexagonLinGr.pdf}\\\end{array}$&	$  \begin{array}{c} 220 \\ \ZZ_2\times\ZZ_2\times \\(\ZZ_5\ltimes\ZZ_{11})\\\end{array}$ & $\begin{array}{c} 2[1]+1[22]+2[110]\\2[1]+1[22]+2[110]\\2[1]+1[22]+2[110]\\\end{array}$	 & $\frac{266}{1} + \frac{1}{5}$	& $(3,3,3,3,3)$\\\hline
\textbf{VIII} 	&$\begin{array}{c}\includegraphics[width=1cm,height=1cm]{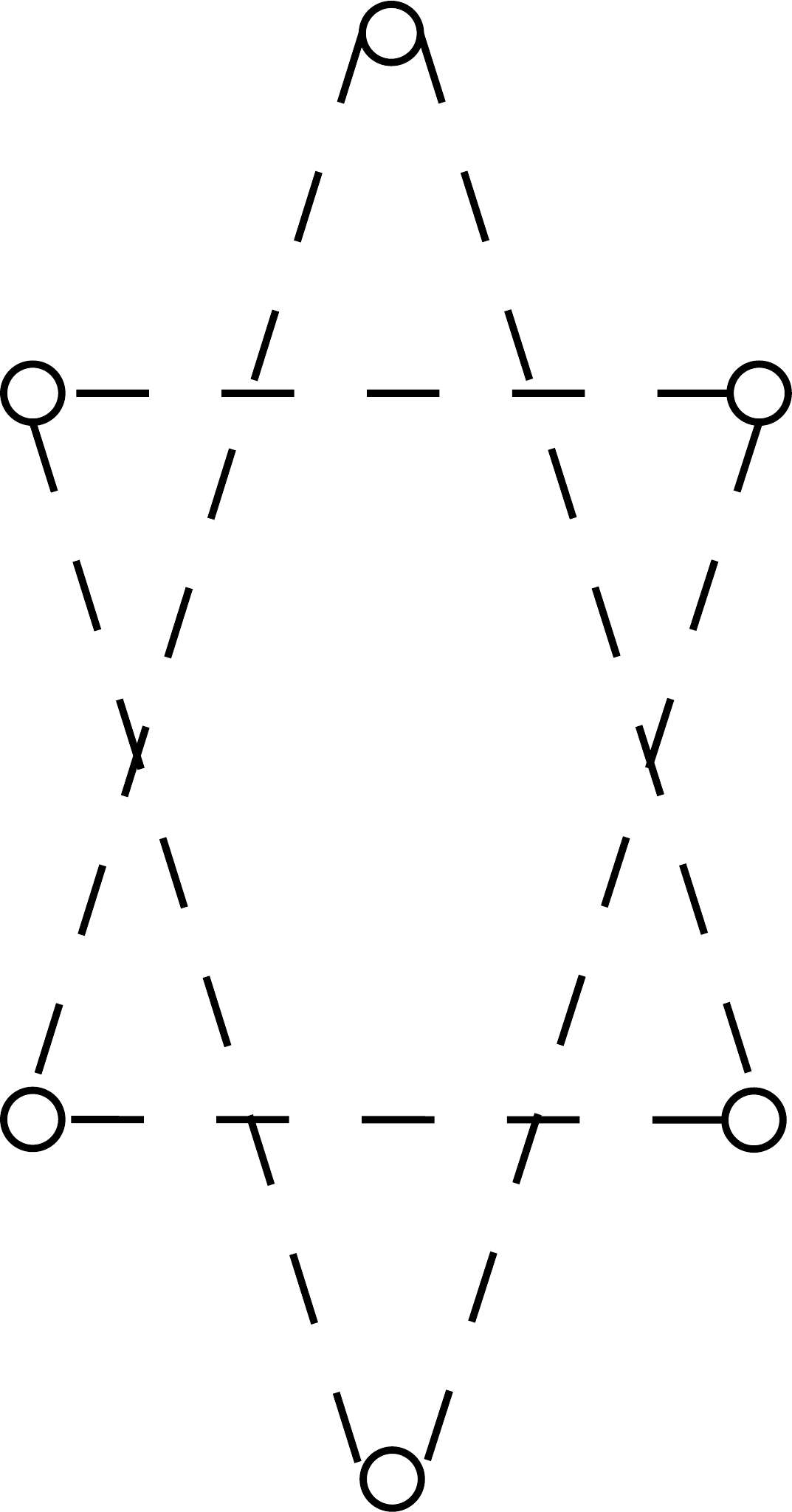}\\\end{array}$&	$ \begin{array}{c}
4
\\
\ZZ_2\times\ZZ_2\\\end{array}$ & $\begin{array}{c}
2[1]+121[2]\\
2[1]+121[2]\\
2[1]+121[2]\\\end{array}$
	 & $\frac{14641}{1}$& $(3,3,3,3,3)$\\\hline
\textbf{IX} 	&$\begin{array}{c}\includegraphics[width=1cm,height=1cm]{hexagonTriangTr.pdf}\\\end{array}$&		$ \begin{array}{c}
4
\\
\ZZ_2\times\ZZ_2\\\end{array}$ & $\begin{array}{c}
2[1]+121[2]\\
2[1]+121[2]\\
2[1]+121[2]\\\end{array}$
	 &$\frac{14641}{1}$	& $(3,3,3,3,3)$\\\hline
\end{tabular}
\caption{{Division algebras with $243$ elements and their properties}}
\label{numerazos}\end{table}%

\section*{Acknowledgments}

This work has been partially supported by MEC - MTM - 2010 - 18370 - C04 - 01, IB08-147, MEC-TIN-2007-61273 and
MICINN-TIN-2010-14971. The authors thankfully acknowledge the computer
resources, technical expertise and assistance provided by the
\emph{{Centro de Supercomputaci\'on y Visualizaci\'on de Madrid
(CeSViMa)}} and the Spanish Supercomputing Network.

\end{document}